\documentclass[11pt]{amsart}
\usepackage{amssymb, latexsym}
\usepackage{graphicx}
\usepackage{placeins}
\theoremstyle{plain}
\newtheorem{theorem}{Theorem}
\newtheorem{corollary}{Corollary}
\newtheorem*{thm*}{Theorem}

\numberwithin{equation}{section}

\begin{document}

\title[Random sequential adsorption  and its continuum limit ]{  A new look at some aspects of one-dimensional  random sequential adsorption  and its continuum limit  }

\author{Ross G. Pinsky}


\address{Department of Mathematics\\
Technion---Israel Institute of Technology\\
Haifa, 32000\\ Israel}
\email{ pinsky@technion.ac.il}

\urladdr{https://pinsky.net.technion.ac.il/}

\date{}

\begin{abstract}
Fix a positive integer $k\ge2$, and for  $n\ge k$, consider a row of $n$ molecules. From among the $n-k+1$ nearest-neighbor $k$-tuples of molecules, select one   uniformly at random and bond the $k$ molecules.
Now, from all the remaining nearest-neighbor $k$-tuples, again select one uniformly at random and bond the $k$ molecules. Continue like this until there are no nearest-neighbor $k$-tuples left.
Let $M^{(n)}_k$ denote the expected value of the  number of bonded molecules. An explicit integral formula for   $m_k:=\lim_{n\to\infty}\frac{M^{(n)}_k}n$ is known,  and an explicit formula for
$m_\infty:=\lim_{k\to\infty}m_k$ is known.
The constant $m_\infty$, known as the R\'enyi parking constant,  arises as the limiting packing density for a continuous analog of the above discrete packing problems. These are all models of what is called
random sequential adsorption (RSA). The first part of this paper studies the gaps of sizes $0,1,\cdots, k-1$ that arise between bonded $k$-tuples
and shows that
after scaling the $k$-grid,  when $k\to\infty$ the
empirical  distribution of expected gaps in the discrete problem on the lattice
 converges weakly to an appropriate gap distribution that is known to hold for the above noted continuous analog.
The second part of the this paper
considers  two different models of the discrete bonding problem when both $k_1$-bonding  and $k_2$-bonding occur, with $2\le k_1<k_2$. Explicit formulas are obtained for the analogs of $m_k$, and  the asymptotic behavior of these analogs is studied both when  $k_2\to\infty$ with $k_1$ fixed, and when $k_1,k_2\to\infty$ at certain ratios.

\end{abstract}

\maketitle
\section{Introduction and Statement of Results}\label{intro}
Fix a positive integer $k\ge2$, and for  $n\ge k$, consider a row of $n$ molecules. From among the $n-k+1$ nearest-neighbor $k$-tuples of molecules, select one   uniformly at random and bond the $k$ molecules.
Now, from all the remaining unbonded nearest-neighbor $k$-tuples, again select one uniformly at random and bond the $k$ molecules. Continue like this until there are no unbonded nearest-neighbor $k$-tuples left.
Let $M^{(n)}_k$ denote the expected value of the  number of bonded molecules and consider $\lim_{n\to\infty}\frac{M^{(n)}_k}n$, the limiting expected density of bonded molecules as the number of molecules increases to infinity.
The following theorem is known.
\begin{theorem}\label{mkthm}
Let $M_k^{(n)}$ denote the expected number of bonded molecules under $k$-bonding on a row of $n$ molecules. Then
\begin{equation}\label{mk}
m_k:=\lim_{n\to\infty}\frac{M^{(n)}_k}n=k\int_0^1\exp(2\sum_{j=1}^{k-1}\frac{s^j-1}j)\thinspace ds.
\end{equation}
\end{theorem}
The case $k=2$, proved by Flory \cite{F39} in 1939, thirty-five years before he won the Nobel Prize in Chemistry, appeared in a chemistry journal,  and was then rediscovered and proved along with a corresponding weak law of large numbers  by Page \cite{P59} in 1959 in a statistics journal. These proofs do not generalize to $k>2$. We proved the  general case   along with a corresponding weak law of large numbers in \cite{P14}.
It turns out though that there were a number of earlier  works that obtained the result \eqref{mk}, with varying levels of rigor and varying levels of detail; see for example, \cite{GHH}, \cite{M}, \cite{FRM}.
The integral in \eqref{mk} can be calculated explicitly only when $k=2$; one has $m_2=1-e^{-2}\approx0.865$.
The values of $m_k$ for various choices of $k$  appear in   the first column of table \ref{table:2} below. It is clear from numerical evidence that $m_k$ is decreasing in $k$; this was stated as a conjecture in \cite{P14}, and seems to be open still.

One can show that
\begin{equation}\label{minfty}
m_\infty:=\lim_{k\to\infty}m_k=\int_0^\infty\exp\left(-2\int_0^x\frac{1-e^{-y}}ydy\right)dx\approx0.7476.
\end{equation}
The constant $m_\infty$ is known as the R\'enyi parking constant. It arises directly from the following continuous packing problem.
Consider the interval $[0,L]$. A car of length 1 wants to park on the interval. If $L<1$, then it can't park. If $L\ge1$, the car chooses a point $p$ uniformly at random from the interval $[0,L-1]$ and parks, taking up the interval $[p,p+1]$.
This creates two smaller vacant intervals---$[0,p]$ and $[p+1,L]$. Now implement the above rule for each of these intervals. Continue like this until no more cars can park. Let $C^{(L)}$ denote the expected value of the length of space taken up by parked cars (or equivalently, of the number of parked cars).
R\'enyi \cite{R58} proved that
the limiting expected density of the space taken up by parked cars as the length of the interval  increases to infinity is given by
$$
\lim_{n\to\infty}\frac{C^{(L)}}L=m_\infty.
$$
In a private communication \cite{B15}, Julien Bureaux provided a proof of  \eqref{minfty}. We will reproduce it below at the end of this introductory section, as variants of it will be needed in this paper.
In a literature search for this paper, we found an older similar proof of \eqref{minfty} in \cite{DPSZ}, which itself cites \cite{GHH}, where the result is demonstrated without  complete mathematical rigor.
If one scales the discrete molecule problem above by $k$, it is easy to see heuristically that the R\'enyi parking problem is the continuum limit of the discrete problem.

Both the discrete and the continuous  models above  are models of what is known in the    chemistry literature as random sequential adsorption (RSA).
For more about models of RSA and other related phenomena, the reader is referred to the review article \cite{E93}.
From the mathematics literature point of view, the above problems are discrete and random packing problems. For a fresh look at the discrete  case with $k=2$, see \cite{G15}.
For some refinements of R\'enyi's result, see for example, \cite{DR}, \cite{SV} and \cite{S24}.

This paper has two contributions.
The first contribution  is a   study  of the gaps between bonded $k$-tuples in the discrete model and a proof that after scaling the one-dimensional integer lattice by $k$,
the
empirical  distribution of expected gaps in the discrete problem on the lattice
 converges weakly  as $k\to\infty$  to an appropriate gap distribution that is known to hold for the R\'enyi parking problem.
The second and more novel contribution is a study of two different models of the discrete bonding problem when both $k_1$-bonding  and $k_2$-bonding occur, with $2\le k_1<k_2$.
For the proofs of the results concerning  one of these two models, the results on the gaps will play an essential role.

Turning to the first topic,
 we now describe the gaps.
When $k$-bonding is completed on an interval of length $n$, there will be between any two consecutive bonded $k$-tuples  a set of $l$ unbonded molecules, for some $l\in\{0,1, \cdots, k-1\}$.
Such an $l$-tuple of unbonded molecules  also occurs before the leftmost bonded $k$-tuple and after the rightmost one.
For $l\in\{0,1,\cdots, k-1\}$, let $X^{(n)}_{k;l}$ denote the random variable counting the number of $l$ gaps, and let $G^{(n)}_{k;l}=EX^{(n)}_{k;l}$  denote the expected value of the number of  $l$-gaps.
(We emphasize that this  counts the number of $l$-gaps, not the number of molecules contained in the $l$-gaps, which is $l$ times as large. Also, we are working with maximal gaps. That is, if there are $l$ molecules between two consecutive $k$ bonds, then this counts only as one $l$ gap and not also as, say, two $l-1$ gaps.)
A central limit theorem was proven for $X^{(n)}_{k;l}$ in \cite{KR}. That paper also proved the following theorem concerning
 $\lim_{n\to\infty} \frac{G^{(n)}_{k;l}}n$, the limiting expected  number  of $l$-gaps per unit length, that is, the limiting expected density of $l$-gaps,  as the number of molecules increases to infinity.
For completeness, we will provide a different but similar  proof of the theorem, which is in the spirit of the proofs of the  rest of the results in this paper, and which we proved before we were aware of \cite{KR}.
\begin{theorem}\label{gapsthm}
Let $G^{(n)}_{k;l}$ denote the expected number of $l$ gaps  under $k$-bonding on a row of $n$ molecules. Then
\begin{equation}\label{lgaps}
g_{k;l}:=\lim_{n\to\infty}\frac{G^{(n)}_{k;l}}n=2\int_0^1(1-s)s^l\exp\left(2\sum_{j=1}^{k-1}\frac{s^j-1}j\right)\thinspace ds,\ l=0,1,\cdots, k-1.
\end{equation}
\end{theorem}
\noindent \bf Remark.\rm\ For $l=k-1$, the integral above can be calculated explicitly.
One has $\left((1-s)^2\exp\left(2\sum_{j=1}^{k-1}\frac{s^j}j\right)\right)'=-2(1-s)s^{k-1}\exp\left(2\sum_{j=1}^{k-1}\frac{s^j}j\right)$, thus from \eqref{lgaps},
$g_{k;k-1}=\exp\left(-2\sum_{j=1}^{k-1}\frac1j\right)$.

Since $M_k^{(n)}$ is the expected number of molecules bonded in $k$-tuples, and since $\frac{M_k^{(n)}}n$ is the number of bonded $k$-tuples, it follows
from the definitions that
$$
\begin{aligned}
&\sum_{l=0}^{k-1}G^{(n)}_{k;l}=\frac1kM^{(n)}_k+1;\\
&\sum_{l=1}^{k-1}lG^{(n)}_{k;l}+M^{(n)}_k=n,
\end{aligned}
$$
and consequently by Theorem \ref{gapsthm},
\begin{equation}\label{dandg}
\begin{aligned}
&\sum_{l=0}^{k-1}g_{k;l}=\frac{m_k}k;\\
&\sum_{l=1}^{k-1}lg_{k;l}=1-m_k.
\end{aligned}
\end{equation}
We remind the reader that $g_{k;l},lg_{k;l}, \frac{m_k}k$  and $m_k$ are respectively the number of $l$-gaps, the number of molecules in $l$-gaps, the number of  $k$-bonds and the number of molecules in $k$-bonds per unit length.
We illustrate this in Table \ref{table:1} for $k=10$. (Due to roundoff error, there is a  discrepancy in the fourth decimal place  between the sum  of the values  in each of the two   columns    and the corresponding  values appearing in the caption to the table.)
\begin{table}[h!]
\centering
\begin{tabular}{||c c c||}
 \hline
 l& $g_{10;l}$& $lg_{10;l}$ \\ [0.5ex]
 \hline\hline
 0 & 0.0187&0   \\
 1 & 0.0124&0.0124   \\
 2 & 0.0095& 0.0190  \\
 3 & 0.0077&0.0232 \\
 4 & 0.0065&0.0260  \\
 5 & 0.0059&0.0280  \\
 6 & 0.0049&0.0293 \\
 7 & 0.0043&0.0303\\
 8 & 0.0039&0.0310\\
 9 & 0.0035&0.0314\\
 \hline
\end{tabular}
\vspace{1ex}
\caption{Gaps densities from Theorem \ref{gapsthm} for $k=10$;\\  $\sum_{l=0}^9g_{10;l}=\frac{m_{10}}{10}\approx0.0770$, \ \ $\sum_{l=1}^9lg_{10;l}=1-m_{10}\approx0.2304$}
\label{table:1}
\end{table}
It follows from \eqref{lgaps} that $g_{k;l}$ is decreasing for $l\in\{0,\cdots, k-1\}$. Numerical evidence suggests that $lg_{k;l}$ is increasing
for $l\in\{0,\cdots, k-1\}$, but we don't have a proof.

In light of  the first line of \eqref{dandg}, define the scaled  empirical measure of the    expected  gap densities for $k$-bonding by
\begin{equation}\label{muk}
\mu^{\text{gaps}}_k([0,\gamma])=\frac k{m_k}\sum_{l: \frac l{k-1}\le \gamma}g_{k;l},\ \gamma\in[0,1].
\end{equation}
We will prove the following weak convergence result for the sequence $\{\mu^{\text{gaps}}_k\}_{k=2}^\infty$ of probability measures on $[0,1]$.
\begin{theorem}\label{weakconvgapsthm}
Consider the probability  measures $\{\mu^{\text{gaps}}_k\}_{k=2}^\infty$ defined in \eqref{muk}, where $g_{k;l}$ is as in \eqref{lgaps} and $m_k$ is as in \eqref{mk}. Then
\begin{equation}\label{weakconvgaps}
\lim_{k\to\infty}\mu^{\text{gaps}}_k([0,\gamma])=\frac2{m_\infty}\int_0^\infty\left(1-e^{-\gamma x}\right)\exp\left(-2\int_0^x\frac{1-e^{-y}}ydy\right)dx, \ \gamma\in[0,1],
\end{equation}
where $m_\infty$ is the R\'enyi parking constant from \eqref{minfty}. In particular, the limiting probability measure
\begin{equation}\label{gapmeasinfty}
\mu^{\text{gaps}}_\infty:=\text{w}-\lim_{k\to\infty}\mu^{\text{gaps}}_k
\end{equation}
has density
\begin{equation}\label{densitygaps}
f_{\mu^{\text{gaps}}_\infty}(\gamma)=\frac2{m_\infty}\int_0^\infty xe^{-\gamma x}\exp\left(-2\int_0^x\frac{1-e^{-y}}ydy\right)dx, \ \gamma\in[0,1].
\end{equation}
\end{theorem}
From \eqref{densitygaps}, it is clear that $\lim_{\gamma\to0}f_{\mu^{\text{gaps}}_\infty}(\gamma)=\infty$.
The graph of the density $f_{\mu^{\text{gaps}}_\infty}$ appears in figure \ref{fig:density}.
We have the following corollary.
\begin{corollary}\label{expvalf}
The expected value of a $\mu^{\text{gaps}}_\infty$-distributed random variable is given by
\begin{equation}\label{gammafgamma}
\int_0^1\gamma f_{\mu^{\text{gaps}}_\infty}(\gamma)d\gamma=\frac{1-m_\infty}{m_\infty}\approx0.3376.
\end{equation}

\end{corollary}
\bf\noindent Remark.\rm\ The corollary indicates that for $k$-bonding with very large $k$, the average size of a gap is just slightly more than $\frac13$ the maximum gap size of $k-1$.
\begin{proof}
Using   \eqref{weakconvgaps} and \eqref{muk} for the first equality below and the second line of \eqref{dandg} for the second equality below, we have
$$
\int_0^1\gamma f_{\mu^{\text{gaps}}_\infty}(\gamma)d\gamma=\lim_{k\to\infty}\frac k{m_k}\sum_{l=1}^{k-1}\frac l{k-1}g_{k;l}=\frac{1-m_\infty}{m_\infty}.
$$
\end{proof}

 In the R\'enyi parking problem on an interval of length $L$, gaps between parked cars are formed, with gap lengths in the interval $[0,1)$.
 B\'ank\"ovi \cite{B62} carefully constructed the appropriate probability space for this problem
and then defined a gap chosen uniformly at random from all of the existing gaps.  The distribution of this gap is of course a probability measure on $[0,1]$. He showed that this gap distribution converges
weakly as $L\to\infty$ to the probability measure $\mu^{\text{gaps}}_\infty$ defined in \eqref{gapmeasinfty}.
Thus, Theorem \ref{weakconvgapsthm} shows that the appropriately scaled  empirical  distribution of expected gaps in the discrete problem on the lattice converges as $k\to\infty$ to the corresponding appropriate object for the R\'enyi problem on the line.

\begin{figure}[htbp]
    \centering
    \includegraphics[scale=.20]{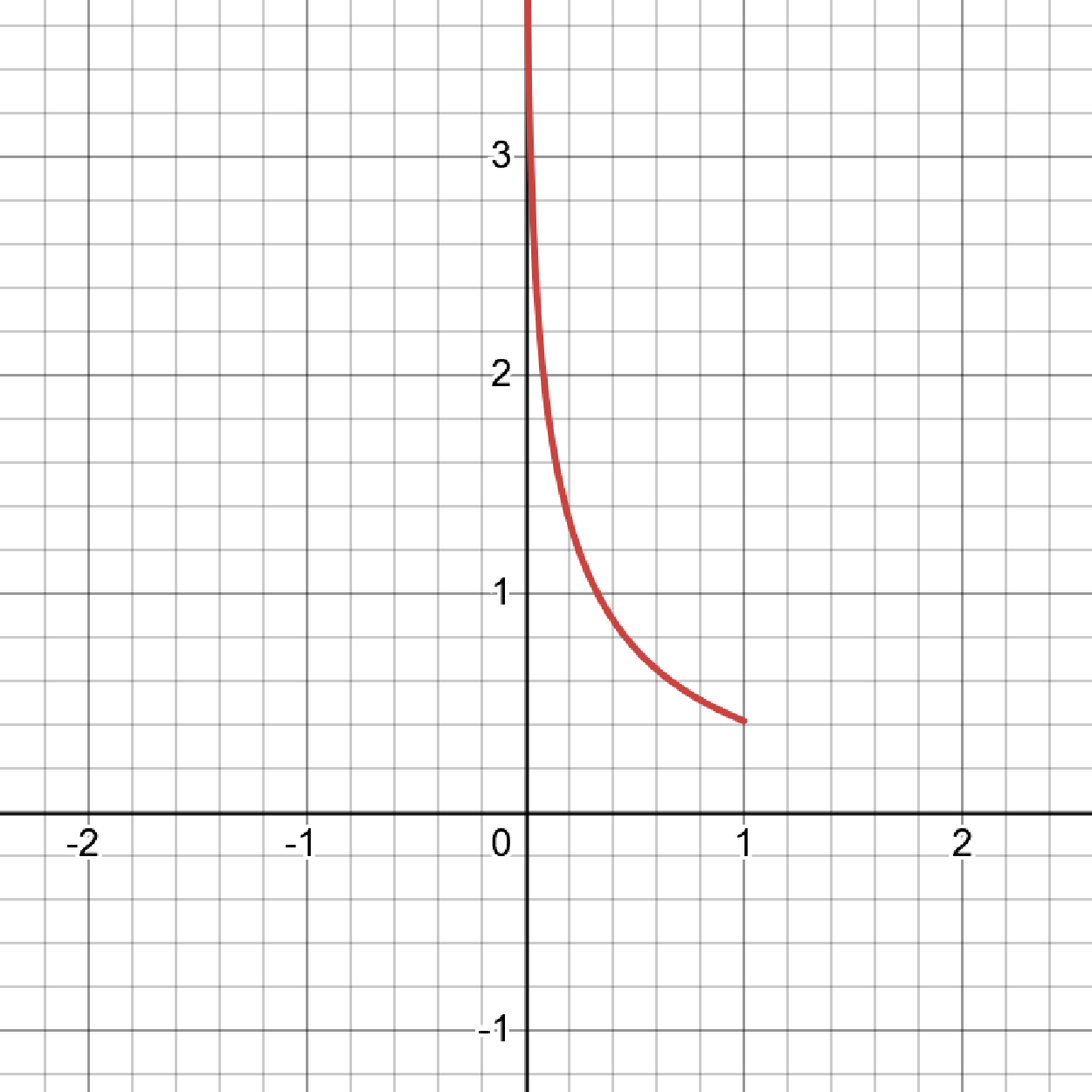}
    \caption{The density $f_{\mu^{\text{gaps}}_\infty}$ of $\mu^{\text{gaps}}_\infty$ from Theorem \ref{weakconvgapsthm}.}
    \label{fig:density}
\end{figure}
\medskip

We now turn to
 a study of the discrete bonding problem when both $k_1$-bonding  and $k_2$-bonding occur, with $2\le k_1<k_2$.
We will consider two different models. In the first model, $k_2$ bonding is performed to its completion on a row of $n$ molecules. This leaves gaps of sizes between 0 and $k_2-1$. Now $k_1$-bonding is implemented on these gaps.
Then as before, we let $n\to\infty$.
In the second model,  at first $k_1$ and $k_2$ bonding occur together. At each step, one considers all of the possible $k_1$-tuples and all of the possible $k_2$-tuples, chooses one of them uniformly at random and implements the bond. When there are  no longer  any unbonded $k_2$-tuples,   $k_1$-bonding continues on all the remaining gaps, whose lengths are of course at most $k_2-1$. Then we let $n\to\infty$.
Both models can be considered to have two stages. In model I, the first stage is $k_2$-bonding and the second stage is $k_1$-bonding on the remaining gaps of lengths less than $k_2$. In model II, the first stage is competitive $k_1$ and $k_2$-bonding, continuing until there are no longer any unbonded $k_2$-tuples. The second stage is $k_1$-bonding on the remaining gaps of lengths less than $k_2$.
We will be interested in results for $k_1$ and $k_2$ fixed, for $k_1$  fixed and $k_2\to\infty$, and for $k_1,k_2\to\infty$ at certain relative rates.
Below we display the definitions of the two models and define the notation for the expected values of the number of bonded molecules:
\begin{equation}\label{models}
\begin{aligned}
&\text{\bf Model I\rm: Let}\ 2\le k_1<k_2\le n. \ \text{Implement}\ k_2\text{-bonding on the line of }\ n\\
&\text{molecules, and when it is completed, implement }\ k_1\text{-bonding on
the }\\
&\text{resulting gaps. Let}\  M^{(n)}_{k_1,k_2;I}\ \text{ denote the expected number of bonded}\\
&\text{molecules. Now let}\ n\to\infty.\\
&\text{\bf Model II\rm: Let}\ 2\le k_1<k_2\le n. \ \text{Implement}\ k_1\ \text{and}\ k_2\text{-bonding together on}\\
&\text{the line of }\ n\
\text{molecules---at each step choose uniformly at random from }\\
&\text{all of the possible}\ k_1\text{-tuples and }\ k_2\text{-tuples, and bind the molecules. When}\\
&\text{there are no longer any unbonded}\ k_2\text{-tuples, continue with }\ k_1\text{-bonding.}\\
&\text{Let}\  M^{(n)}_{k_1,k_2;II}\ \text{ denote the expected number of bonded molecules, and let }\\
&M^{(n)}_{k_1,k_2;II_1}\ \text{denote the expected number of bonded molecules at the end of}\\
&\text{stage one, when there are no longer any unbonded}\ k_2\text{-tuples. Now let}\ n\to\infty.
\end{aligned}
\end{equation}
\bf\noindent Remark.\rm\ Note that the parallel of $M^{(n)}_{k_1,k_2;II_1}$ for model I is $M^{(n)}_{k_2}$, the expected number of bonded molecules for $k_2$-bonding on a row of $n$ molecules.

We now consider $\lim_{n\to\infty}\frac{M^{(n)}_{k_1,k_2;I}}n$ and $\lim_{n\to\infty}\frac{M^{(n)}_{k_1,k_2;II}}n$, the limiting expected density of bonded molecules as the number of molecules increases to infinity in the two models, and also
$\lim_{n\to\infty}\frac{M^{(n)}_{k_1,k_2;II_1}}n$, the  limiting expected density of bonded molecules at the end of stage one in model II, as the number of molecules increases to infinity. Of course, the limiting expected density of bonded molecules at the end of stage one in model I as the number of molecules increases to infinity is
$\lim_{n\to\infty}\frac{M^{(n)}_{k_2}}n=m_{k_2}$, which appears in Theorem \ref{mkthm}.

For Model I, we have the following theorem, whose proof is immediate in light of Theorems \ref{mkthm} and \ref{gapsthm}.
\begin{theorem}\label{k1k2Ithm}
Let $M^{(n)}_{k_1,k_2;I}$ denote the expected number of bonded molecules under Model I in \eqref{models}.
Then
\begin{equation}\label{k1k2I}
\begin{aligned}
&m_{k_1,k_2;I}:=\lim_{n\to\infty}\frac{M^{(n)}_{k_1,k_2;I}}n=m_{k_2}+\sum_{l=k_1}^{k_2-1}g_{k_2;l}M_{k_1}^{(l)},
\end{aligned}
\end{equation}
where $m_{k_2}$ is as in Theorem \ref{mkthm},  $g_{k;m}$ is as in Theorem \ref{gapsthm} and $M_{k_1}^{(l)}$ is the expected number of bonded molecules under $k_1$-bonding on a row of $l$ molecules.
\end{theorem}
\begin{proof}
In Model 1, first $k_2$-bonding is implemented. When it is completed, between any two consecutive  $k$-tuples, there will be a gap of size
 of between 0 and  $k_2-1$. On these gaps, $k_1$-bonding is implemented. The proof follows from this description along with Theorems \ref{mkthm} and \ref{gapsthm} and the linearity of the expectation.
\end{proof}
We will prove the following theorem for Model II.
\begin{theorem}\label{k1k2IIthm}
Let $M^{(n)}_{k_1,k_2;II}$ denote the expected number of bonded molecules under Model II in \eqref{models},
and let $M^{(n)}_{k_1,k_2;II_1}$ denote the expected number of bonded molecules at the end of stage one, when there are no longer any unbonded $k_2$-tuples.
Then
\begin{equation}\label{k1k2II1}
\begin{aligned}
&m_{k_1,k_2;II_1}:=\lim_{n\to\infty}\frac{M^{(n)}_{k_1,k_2;II_1}}n=\\
&\frac12\int_0^1s^{\frac{k_2-k_1}2}\big(k_1+k_2+k_1(k_2-k_1)(1-s)\big)\exp\left(\sum_{j=1}^{k_1-1}\frac{s^j-1}j+\sum_{j=1}^{k_2-1}\frac{s^j-1}j\right) ds,
\end{aligned}
\end{equation}
and
\begin{equation}\label{k1k2II}
m_{k_1,k_2;II}:=\lim_{n\to\infty}\frac{M^{(n)}_{k_1,k_2;II}}n=m_{k_1,k_2;II_1}+
\int_0^1B(s)\exp\left(\sum_{j=1}^{k_1-1}\frac{s^j-1}j+\sum_{j=1}^{k_2-1}\frac{s^j-1}j\right)ds,
\end{equation}
where
\begin{equation}\label{B(s)}
\begin{aligned}
&B(s)=(1-s)^2s^{-\frac12k_1-\frac12k_2}\left(\sum_{n=\max(k_2,2k_1)}^{k_2+k_1-1}\big(\sum_{j=k_1}^{n-k_1}M^{(j)}_{k_1}\big)s^n+\sum_{n=k_2+k_1}^{2k_2-1}\big(\sum_{j=k_1}^{n-k_2}M^{(j)}_{k_1}\big)s^n\right)+\\
&(1-s)\left(s^{\frac12k_2+\frac12k_1}+s^{\frac32k_2-\frac12k_1}\right)\left(\sum_{j=k_1}^{k_2-1}M^{(j)}_{k_1}\right),
\end{aligned}
    \end{equation}
and $M_{k_1}^{(j)}$ is the expected number of bonded molecules under $k_1$-bonding on a row of $j$ molecules.
\end{theorem}
\bf\noindent Remark.\rm\ The  explicit integral expressions obtained in Theorems \ref{k1k2Ithm} for $m_{k_1,k_2;I}$ and in  Theorem \ref{k1k2IIthm} for $m_{k_1,k_2;II}$ depend on the values
of $M_{k_1}^{(j)}$, for $j\in\{k_1,\cdots, k_2-1\}$. (In Theorem \ref{k1k2Ithm} the integral expressions arise through the terms $g_{k_2;l}$.) For small $k_2$, these values can be figured out  directly. In the case of general $k_2$, if $k_2\le 2k_1$, then from the definition of the $k_1$-bonding model, all of these values are immediate: $M^{(j)}_{k_1}=k_1,\ j=k_1,\cdots, k_2-1$.

Table \ref{table:2}    compares the values $m_{k_1}, m_{k_1,k_2,I}$ and $m_{k_1,k_2,II}$ for various values of $k_1$, with $k_2=k_1+1$ and with $k_2=2k_1$.
We conjecture that $m_{k_1;k_2;I}$ and $m_{k_1,k_2;II}$ are increasing in 
$k_2$ for $k_2>k_1\ge 2$.

\begin{table}[h!]
\centering
\begin{tabular}{||c c c  c c c||}
 \hline
 $k_1$& $m_{k_1}$& $m_{k_1,k_1+1;I}$&$m_{k_1,2k_1;I}$& $m_{k_1,k_1+1;II}$&$m_{k_1,2k_1;II}$\\ [0.5ex]
 \hline\hline
 2&0.865&0.923&0.924&0.892&0.894\\
3&0.824&0.881&0.900&0.850&0.860\\
5&0.793&0.837&0.880&0.813&0.833\\
10&0.770&0.796&0.866&0.782&0.814\\
100&0.750&0.753&0.854&0.751&0.796\\
 \hline
\end{tabular}
\vspace{1ex}
\caption{Limiting bonding densities under  various regimes from Theorems \ref{mkthm}, \ref{k1k2Ithm} and \ref{k1k2IIthm}.}
\label{table:2}
\end{table}

In the remark after Theorem \ref{k1k2IIthm}, it was seen  that one can calculate generically $M_{k_1}^{(j)}$ essentially only when $j\le 2k_1$. (We say ``essentially'' because actually it is easy to figure out by hand
$M_{k_1}^{(j)}$  for $j=2k_1+1$, and then progressively harder for $j=2k_1+2,\cdots$.) But Theorem \ref{mkthm} gives the asymptotic behavior of $M_{k_1}^{(j)}$ as $j\to\infty$.
In the formulas for $m_{k_1,k_2;I}$ and $m_{k_1,k_2;II}$, the influence of $M_{k_1}^{(j)}$ for any particular $j$ becomes negligible as $k_2\to\infty$.
(This can be seen easily in Theorem \ref{k1k2Ithm} since $\lim_{k_2\to\infty}g_{k_2;l}=0$, for any $l$. That this limit is zero can be seen by considering \eqref{gapstoL} with the term $\sum_{l=0}^Ls^l$ replaced by $s^l$.)
This allows us to evaluate the limits $\lim_{k_2\to\infty}m_{k_1,k_2;I}$ and  $\lim_{k_2\to\infty}m_{k_1,k_2;II}$ for any fixed $k_1$.
We can also evaluate  the limits
$\lim_{k_1\to\infty}m_{k_1,Lk_1;I}$
and $\lim_{k_1\to\infty}m_{k_1,Lk_1;II}$ with $L\in(1,2]$. The reason that  $k_2=Lk_1$ is restricted to $L\in(1,2]$ is  the issue noted in the first line of this paragraph.
We begin with the results for fixed $k_1$.
\begin{theorem}\label{mk1k2inftyIthm}
Let $m_{k_1;k_2;I}$  be as in Theorem \ref{k1k2Ithm}.
One has
\begin{equation}\label{mk1k2inftyI}
m_{k_1,\infty;I}:=\lim_{k_2\to\infty}m_{k_1,k_2;I}=m_\infty+(1-m_\infty)m_{k_1}\approx0.7476+0.252\thinspace m_{k_1}.
\end{equation}
Thus also,
\begin{equation*}\label{mdoubleI}
\lim_{k_1\to\infty}m_{k_1,\infty;I}=m_\infty+(1-m_\infty)m_\infty\approx0.9363.
\end{equation*}
\end{theorem}
\bf\noindent Remark.\rm\
Theorem \ref{mk1k2inftyIthm} has a  quick proof using Theorems \ref{mkthm}, \ref{gapsthm} and \ref{k1k2Ithm}, but we don't present it here in order not to interrupt the exposition.
Theorem \ref{mk1k2inftyIthm}
 is easy to explain intuitively. Fix a very large $k_2$ and let $n$, the total number of molecules in the row, be much larger than $k_2$.
  In Model I, first $k_2$-bonding is implemented.
Since $n$ is much larger than $k_2$, by Theorem \ref{mkthm} the percentage of bonded molecules is close to $m_{k_2}$, and since $k_2$ is very large, $m_{k_2}$ is close to
$m_\infty$. This accounts for the term $m_\infty$ in \eqref{mk1k2inftyI}.
Now $k_1$-bonding is implemented on the gaps. The percentage of the original $n$  molecules that are in these gaps  is close to $1-m_{k_2}$, which in turn is close to $1-m_\infty$.
Since the limiting density of gaps $g_{k_2;l}$ as $n\to\infty$ satisfy $\lim_{k_2\to\infty}g_{k_2;l}=0$, for any fixed $l$, it follows that for our very large $n$, after the
$k_2$-bonding is implemented, most of the gaps on which $k_1$-bonding is now  implemented are large.
Thus, by Theorem \ref{mkthm}, the percentage of molecules  that were unbonded after the $k_2$-bonding was completed and which now get bonded during the $k_1$-bonding phase is  close to  $m_{k_1}$. This accounts for the term $m_{k_1}(1-m_\infty)$ in \eqref{mk1k2inftyI}.
\begin{theorem}\label{mk1k2inftyIIthm}
Let $m_{k_1;k_2;II}$ and $m_{k_1;k_2;II_1}$ be as in Theorem \ref{k1k2IIthm}.
One has
\begin{equation}\label{mk1k2inftyII_1}
m_{k_1,\infty;II_1}:=\lim_{k_2\to\infty}m_{k_1,k_2;II_1}=D\approx 0.4166,
\end{equation}
and
\begin{equation}\label{mk1k2inftyII}
m_{k_1,\infty;II}:=\lim_{k_2\to\infty}m_{k_1,k_2;II}=D+(1-D)m_{k_1}\approx 0.4166+0.5834\thinspace m_{k_1},
\end{equation}
where
\begin{equation}\label{D}
D=\frac12\int_0^\infty e^{-\frac12x}e^{-\int_0^x\frac{1-e^{-y}}ydy}dx\approx0.4166.
\end{equation}
Thus also,
\begin{equation*}\label{mdoubleII}
\lim_{k_1\to\infty}m_{k_1,\infty;II}=D+(1-D)m_\infty\approx0.8528.
\end{equation*}
\end{theorem}
\bf\noindent Remark.\rm\ Analogous  to  Theorem \ref{mk1k2inftyIthm} and the remark  following that theorem,   Theorem \ref{mk1k2inftyIIthm} has the following explanation.
In Model II, $k_1$-bonding and $k_2$-bonding  take place in a competing manner until $k_2$ bonding is no longer possible. Then $k_1$-bonding continues alone until it is completed.
Equation \eqref{mk1k2inftyII_1} shows that for fixed $k_1$ and very large $k_2$, at the point that $k_2$-bonding is no longer possible, the percentage of bonded molecules (both $k_1$-bonds and $k_2$-bonds) is approximately $D$. After this, $k_1$ bonding continues to be implemented on the remaining molecules whose percentage is approximately $1-D$.

 We note   two very interesting points concerning the above remark. First of all, the value of $D$ is independent of $k_1$, despite the fact that in Model II, unlike in Model 1 analyzed in Theorem \ref{mk1k2inftyIthm}, $k_1$ and $k_2$ bonding are competing at the same time. Second of all, \it the value of $D$ ($\approx0.4166$) is drastically lower than the value of $m_\infty$ ($\approx0.7476$). This shows that the $k_1$-sized bonds that occur during the competing stage of $k_1$ and
$k_2$ bonding heavily reduce the number of consecutive unbonded $k_2$-tuples available for $k_2$-bonding.\rm\
This carries over heuristically to the continuous parking model of Renyi. Let $k_2$ be very large. Scaling by $k_2$, we think of the $k_2$ bond as being
a clump of bonded molecules of unit length (corresponding to the parked car of length 1). The bonds of length $k_1$, when scaled by $k_2$, now represent minute particles
that interfere with bonding. 

Table \ref{table:3}
compares the values of $m_{k_1}, m_{k_1,\infty;I}$ and $m_{k_1,\infty;II}$ for various values of $k_1$. One might also want to compare Table \ref{table:3} with Table \ref{table:2}.
Tables \ref{table:2} and \ref{table:3} show that for fixed $k_1$, and any choice of $k_2$, or alternatively $k_2\to\infty$, the highest percentage of bonded molecules occurs for Model I $k_1,k_2$-bonding, then for
Model II $k_1,k_2$-bonding, and then the lowest percentage is for the standard $k_1$-bonding.
\begin{table}[h!]
\centering
\begin{tabular}{|| c c c c||}
 \hline
 $k_1$& $m_{k_1}$& $m_{k_1,\infty;I}$&$m_{k_1,\infty;II}$\\ [0.5ex]
 \hline\hline
 2&0.865&0.966&0.921\\
3&0.824&0.956&0.897\\
5&0.792&0.948&0.879\\
10&0.770&0.942&0.866\\
100&0.750&0.937&0.854\\
 \hline
\end{tabular}
\vspace{1ex}
\caption{Limiting bonding densities under $k_1$-bonding and under Model I and Model II $k_1,k_2$-bonding when $k_2\to\infty$,  from Theorems \ref{mkthm}, \ref{mk1k2inftyIthm} and
\ref{mk1k2inftyIIthm}.}
\label{table:3}
\end{table}

We now evaluate $\lim_{k_1\to\infty}m_{k_1,Lk_1;I}$
and $\lim_{k_1\to\infty}m_{k_1,Lk_1;II}$ with $L\in(1,2]$.
\begin{theorem}\label{k1k1LIinftythm}
Let $m_{k_1,k_2;I}$ be as in Theorem \ref{k1k2Ithm}. Then for $L\in (1,2]$,
\begin{equation}\label{k1k1LIinfty}
\begin{aligned}
&m_{\infty,L\infty;I}:=\lim_{k_1\to\infty}m_{k_1,[Lk_1];I}=m_\infty+\frac2L\int_0^\infty\left(e^{-\frac1L x}-e^{-x}\right)e^{-2\int_0^x\frac{1-e^{-y}}ydy}dx.
\end{aligned}
\end{equation}
\end{theorem}
\bf\noindent Remark.\rm\ Note that if one substitutes $L=1$ in \eqref{k1k1LIinfty}, the expression reduces to $m_\infty$  as is to be expected.

As with Theorem \ref{mk1k2inftyIthm}, Theorem \ref{k1k1LIinftythm}   has a rather quick proof using Theorems \ref{mkthm}, \ref{gapsthm} and \ref{k1k2Ithm}, but again we don't present it here in order to not interrupt the exposition .
\begin{theorem}\label{k1k1LIIinftythm}
Let $m_{k_1,k_2;II}$ be as in Theorem \ref{k1k2IIthm}. Then for $L\in (1,2]$,
\begin{equation}\label{k1k1LIIinfty}
\begin{aligned}
&m_{\infty,L\infty;II}:=\lim_{k_1\to\infty}m_{k_1,[Lk_1];II}=\\
&\frac12\int_0^\infty \left(L+1+(L-1)x\right) e^{-\frac{L-1}2 x}e^{-\int_0^x\frac{1-e^{-y}}ydy}e^{-\int_0^{Lx}\frac{1-e^{-y}}ydy}dx+\\
&\int_0^\infty\left(e^{-\frac{3-L}2x}-e^{-\frac{3L-1}2x}\right)e^{-\int_0^x\frac{1-e^{-y}}ydy}e^{-\int_0^{Lx}\frac{1-e^{-y}}ydy}dx.
\end{aligned}
\end{equation}
\end{theorem}
\bf\noindent Remark.\rm\ Note that if one substitutes $L=1$ in \eqref{k1k1LIIinfty}, the expression reduces to $m_\infty$ in \eqref{minfty}, as is to be expected.

Table \ref{table:4} compares the values of $m_{\infty,L\infty;I}$ and $m_{\infty,L\infty;II}$ for various values of $L$.
\begin{table}[h!]
\centering
\begin{tabular}{|| c c  c||}
 \hline
 $L$& $m_{\infty,L\infty;I}$&$m_{\infty,L\infty;II}$\\ [0.5ex]
 \hline\hline
 1.2&0.795&0.768\\
1.4&0.822&0.779\\
1.6&0.838&0.785\\
1.8&0.847&0.790\\
2.0&0.852&0.794\\
 \hline
\end{tabular}
\vspace{1ex}
\caption{Limiting bonding densities under Model I and Model II $k_1,k_2$-bonding with $k_2=Lk_1$ and $k_1\to\infty$,  from Theorems
\ref{k1k1LIinftythm} and \ref{k1k1LIIinftythm}.}
\label{table:4}
\end{table}

It would be interesting to understand the behavior of $\lim_{k_1\to\infty}m_{k_1,Lk_1;I}$ and
$\lim_{k_1\to\infty}m_{k_1,Lk_1;II}$ also for the case $L>2$ as well as for the case that $L=L(k_1)\to\infty$ at various rates.

As noted earlier, we end the introduction with the  proof of \eqref{minfty}.

\noindent \it Proof of \eqref{minfty}.\rm\
Making the substitution $y=k(1-t)$ in the third equality below, we write
\begin{equation}\label{sj-1}
\begin{aligned}
&\sum_{j=1}^{k-1}\frac{s^j-1}j=-\sum_{j=1}^{k-1}\int_s^1t^{j-1}dt=-\int_s^1\frac{1-t^{k-1}}{1-t}dt=-\int_0^{k(1-s)}\frac{1-(1-\frac yk)^{k-1}}ydy.
\end{aligned}
\end{equation}
Therefore, making the substitution $x=k(1-s)$ in the third equality below, we have
\begin{equation}\label{minftyderiv}
\begin{aligned}
&k\int_0^1e^{2\sum_{j=1}^{k-1}\frac{s^j-1}j}ds=k\int_0^1\exp\left(-2\int_0^{k(1-s)}\frac{1-(1-\frac yk)^{k-1}}ydy\right)ds=\\
&\int_0^k\exp\left(-2\int_0^x\frac{1-(1-\frac yk)^{k-1}}ydy\right)dx\stackrel{k\to\infty}{\to}\int_0^\infty\exp\left(-2\int_0^x\frac{1-e^{-y}}ydy\right)dx,
\end{aligned}
\end{equation}
where the limit follows by applying the dominated convergence theorem. This  proves \eqref{minfty}.
\hfill $\square$

We prove Theorem \ref{gapsthm} in section \ref{pfgaps}, Theorem \ref{weakconvgapsthm} in section \ref{weakconvgapsthmpf}, and  Theorems \ref{k1k2IIthm}-\ref{k1k1LIIinftythm} respectively in sections \ref{k1k2IIthmpf}-\ref{k1k1LIIinftythmpf}.
\FloatBarrier

\section{Proof of Theorem \ref{gapsthm}}\label{pfgaps}
In \cite{P14}, we proved Theorem \ref{mkthm} by analyzing the generating function of the sequence $\{S_k^{(n)}\}_{n=1}^\infty$, where $S_k^{(n)}=\sum_{i=1}^nM_k^{(i)}$. The reason we chose to work with the generating function of $\{S_k^{(n)}\}_{n=1}^\infty$ instead of with the generating function of $\{M_k^{(n)}\}_{n=1}^\infty$ is that the former sequence satisfies a three term recurrence formula while the latter sequence satisfies a recursion formula that grows with $n$. This was good enough for the proof of Theorem \ref{mkthm}, but it does not allow for the proof of Theorem \ref{gapsthm}.
The method we employ here can be slightly modified to prove Theorem \ref{mkthm}, as we point out below after completing the proof of Theorem \ref{gapsthm}.

As noted in the introduction, after $k$-bonding on a row of $n$ molecules is completed, between every two consecutive bonded $k$-tuples there will be  a gap of $l$ unbonded molecules, for some
$l\in\{0,1,\cdots, k-1\}$.  There will also be such gaps before the leftmost bonded $k$-tuple and after the rightmost bonded $k$-tuple.
For $n\ge1$, let $G_{k;l}^{(n)}$ denote the expected number of such gaps of length $l$. Define $G_{k;l}^{(0)}=0$, if $l\neq0$ and $G_{k;0}^{(0)}=1$. From the definition of $k$ bonding on a row of $n$ molecules, we have
\begin{equation}\label{gaprecur}
\begin{aligned}
&G_{k;l}^{(n)}=\frac1{n-k+1}\sum_{j=1}^{n-k+1}\left(G_{k;l}^{(j-1)}+G_{k;l}^{(n+1-j-k)} \right)=\\
&\frac2{n-k+1}\sum_{j=0}^{n-k}G_{k;l}^{(j)},\ n\ge k,
\end{aligned}
\end{equation}
with the boundary condition
\begin{equation}\label{gaprecurbc}
G_{k;l}^{(n)}=\begin{cases}1, n=l;\\0, n\in\{0,\cdots, k-1\}-\{l\}.\end{cases}
\end{equation}
Multiplying both sides of \eqref{gaprecur} by $(n-k+1)t^n$ and summing over $n$ from $k$ to $\infty$, we obtain
\begin{equation}\label{multtn}
\sum_{n=k}^\infty(n-k+1)G_{k;l}^{(n)}t^n=2\sum_{n=k}^\infty\left(\sum_{j=0}^{n-k}G_{k;l}^{(j)}\right)t^n.
\end{equation}

Define the generating function
\begin{equation}\label{ggenfunc}
g_{k;l}(t)=\sum_{n=k}^\infty G_{k;l}^{(n)}t^n.
\end{equation}
Using the boundary condition  \eqref{gaprecurbc}, we write the right hand side of \eqref{multtn} as
\begin{equation}\label{doublesum}
\begin{aligned}
&\sum_{n=k}^\infty\left(\sum_{j=0}^{n-k}G_{k;l}^{(j)}\right)t^n=\sum_{n=k}^{2k-1}\left(\sum_{j=0}^{n-k}G_{k;l}^{(j)}\right)t^n+
\sum_{n=2k}^\infty\left(\sum_{j=0}^{n-k}G_{k;l}^{(j)}\right)t^n=\\
&\sum_{n=k+l}^{2k-1}t^n+\sum_{n=2k}^\infty\left(1+\sum_{j=k}^{n-k}G_{k;l}^{(j)}\right)t^n=\sum_{n=k+1}^\infty t^n+\sum_{n=2k}^\infty\left(\sum_{j=k}^{n-k}G_{k;l}^{(j)}\right)t^n.
\end{aligned}
\end{equation}
We have
\begin{equation}\label{tnfree}
t^k\thinspace\frac{g_{k;l}(t)}{1-t}=t^k\left(\sum_{r=0}^\infty t^r\right)\left(\sum_{j=k}^\infty G_{k;l}^{(j)}t^j\right)=\sum_{n=2k}^\infty\left(\sum_{j=k}^{n-k}G_{k;l}^{(j)}\right)t^n.
\end{equation}
From \eqref{multtn}-\eqref{tnfree}, we conclude that
$$
tg_{k;l}'(t)-(k-1)g_{k;l}(t)=\frac{2t^{k+l}}{1-t}+\frac{2t^kg_{k;l}(t)}{1-t},
$$
which we write as
\begin{equation}\label{ode}
\begin{aligned}
&g_{k;l}'(t)=A(t)g_{k;l}(t)+B(t),\ \text{where}\\
&A(t)=\frac{k-1}t+\frac{2t^{k-1}}{1-t};\\
&B(t)=\frac{2t^{k+l-1}}{1-t},
\end{aligned}
\end{equation}
 the dependence of $A$ on $k$ and of $B$ on $k$ and $l$ being suppressed.
Solving the linear ODE gives for any $\epsilon\in(0,1)$,
\begin{equation}\label{odesolu}
g_{k;l}(t)=g_{k;l}(\epsilon)e^{\int_\epsilon^tA(s)ds}+\int_\epsilon^te^{\int_s^tA(r)dr}B(s)ds,\ \epsilon<t<1.
\end{equation}

From \eqref{ode}, we write
$$
A(t)=\frac{k-1}t+\frac2{1-t}-2\sum_{i=0}^{k-2}t^i,
$$
from which it follows that
\begin{equation}\label{eAform}
e^{\int_s^tA(r)dr}=\frac{t^{k-1}}{s^{k-1}}\frac{(1-s)^2}{(1-t)^2}e^{2\sum_{j=1}^{k-1}\frac{s^j-t^j}j}.
\end{equation}
By \eqref{ggenfunc},  $g_{k;l}(\epsilon)=O(\epsilon^k)$, thus  it follows from \eqref{eAform}  that $\lim_{\epsilon\to0}g(\epsilon)e^{\int_\epsilon^tA(s)ds}=0$. Thus, from \eqref{ode}-\eqref{eAform} we conclude that
\begin{equation}\label{g(t)}
g_{k;l}(t)=\frac{2t^{k-1}}{(1-t)^2}\int_0^t(1-s)s^le^{2\sum_{j=1}^{k-1}\frac{s^j-t^j}j}ds.
\end{equation}

It is easy to see that if $\lim_{n\to\infty}\frac{G_{k;l}^{(n)}}n=a$, then
$\lim_{t\to1^-}(1-t)^2g_{k;l}(t)=a$.
Indeed, for any $\epsilon>0$, there exists an $N_\epsilon$ such that $a-\epsilon\le \frac{G_{k;l}^{(n)}}n\le a+\epsilon$, for $n\ge N\epsilon$.
Write
\begin{equation}\label{1limpf}
\sum_{n=k}^{N_\epsilon-1}G_{k;l}^{(n)}t^n+\sum_{n=N_\epsilon}^\infty (a-\epsilon)nt^n
\le g_{k;l}(t)\le \sum_{n=k}^{N_\epsilon-1}G_{k;l}^{(n)}t^n+\sum_{n=N_\epsilon}^\infty (a+\epsilon)nt^n.
\end{equation}
One has
\begin{equation}\label{2limpf}
\sum_{n=N_\epsilon}^\infty bnt^n=bt\left(\frac{t^{N_\epsilon}}{1-t}\right)'=bt\left(\frac{t^{N_\epsilon}}{(1-t)^2}+\frac{N_\epsilon t^{N_\epsilon-1}}{1-t}\right), \ \text{for}\ b>0.
\end{equation}
From \eqref{1limpf} and \eqref{2limpf}, it follows that
$$
a-\epsilon\le \liminf_{t\to1^-}(1-t)^2g_{k;l}(t)\le\limsup_{t\to1^-}(1-t)^2g_{k;l}(t)\le a+\epsilon.
$$
Unfortunately, we don't see any very quick way to prove that $\lim_{n\to\infty}\frac{G_{k;l}^{(n)}}n$ exists.
Define $S_{k;l}^{(n)}=\sum_{j=0}^nG_{k;l}^{(j)}$.
From \eqref{gaprecur}, one has
$$
G_{k;l}^{(n)}=\frac2{n-k+1}S_{k;l}^{(n-k)}.
$$
Thus,
$\lim_{n\to\infty}\frac{G_{k;l}^{(n)}}n$ exists if and only if
$\lim_{n\to\infty}\frac{S_{k;l}^{(n)}}{n^2}$ exists.
The proof that this latter limit exists is almost exactly the same as the proof of Proposition 3.3 in \cite{P14}.
In light of this, we conclude from
\eqref{g(t)} that
$$
\lim_{n\to\infty}\frac{G_{k;l}^{(n)}}n=\lim_{t\to1^-}(1-t)^2g_{k;l}(t)=2\int_0^1(1-s)s^le^{2\sum_{j=1}^{k-1}\frac{s^j-1}j}ds,
$$
which proves Theorem \ref{gapsthm}.\hfill $\square$

We note that the method above can also be use to prove Theorem \ref{mkthm}. Instead of \eqref{gaprecur}, one has
\begin{equation*}
\begin{aligned}
&M_k^{(n)}=\frac1{n-k+1}\sum_{j=1}^{n-k+1}\left(M_k^{(j-1)}+M_k^{(n+1-j-k)}+k \right)=\\
&k+\frac2{n-k+1}\sum_{j=0}^{n-k}M_{k;l}^{(j)},\ n\ge k,
\end{aligned}
\end{equation*}
with boundary condition
$M_k^{(n)}=0$, for $n=0,\cdots k-1$.
\section{Proof of Theorem \ref{weakconvgapsthm}}\label{weakconvgapsthmpf}
  Making the same substitutions used in   \eqref{sj-1} and \eqref{minftyderiv}, and recalling
 \eqref{lgaps}, we have
\begin{equation}\label{gapstoL}
\begin{aligned}
&\sum_{l=0}^Lg_{k;l}=2\int_0^1(1-s)\left(\sum_{l=0}^Ls^l\right)e^{2\sum_{j=1}^{k-1}\frac{s^j-1}j}ds=2\int_0^1(1-s^{L+1})e^{2\sum_{j=1}^{k-1}\frac{s^j-1}j}ds=\\
&\frac2k\int_0^k\left(1-(1-\frac xk)^{L+1}\right)\exp\left(-2\int_0^x\frac{1-(1-\frac yk)^{k-1}}ydy\right)dx,\ L=0,1,2,\cdots.
\end{aligned}
\end{equation}
Recalling the definition of $\mu_k^{\text{gaps}}$ in \eqref{muk}, and using \eqref{gapstoL} for the final equality below, we have
\begin{equation}\label{gapstogammak-1}
\begin{aligned}
&\mu_k^{\text{gaps}}([0,\gamma])=\frac k{m_k}\sum_{l:\frac l{k-1}\le\gamma}g_{k;l}=\frac k{m_k}\sum_{l=0}^{[\gamma(k-1)]}g_{k;l}=\\
&\frac2{m_k}\int_0^k\left(1-(1-\frac xk)^{[\gamma(k-1)]+1}\right)\exp\left(-2\int_0^x\frac{1-(1-\frac yk)^{k-1}}ydy\right)dx, \text{for}\ \gamma\in(0,1].
\end{aligned}
\end{equation}
Theorem \ref{weakconvgapsthm} now follows from \eqref{minfty} and from applying the dominated convergence theorem
 upon letting $k\to\infty$ in  \eqref{gapstogammak-1}.
\hfill $\square$

\section{Proof of Theorem \ref{k1k2IIthm}}\label{k1k2IIthmpf}
From the definition of Model II $k_1,k_2$-bonding, it follows that $M_{k_1,k_2;II}^{(n)}$, the expected value of the number of bonded molecules from a row of $n$ molecules, satisfies
\begin{equation}\label{k1k2recur}
\begin{aligned}
&M_{k_1,k_2;II}^{(n)}=\frac1{(n-k_1+1)+(n-k_2+1)}\times\\
&\left(\sum_{j=1}^{n-k_1+1}\left(M_{k_1,k_2;II}^{(j-1)}+M_{k_1,k_2;II}^{(n-j-k_1+1)}+k_1\right)+
\sum_{j=1}^{n-k_2+1}\left(M_{k_1,k_2;II}^{(j-1)}+M_{k_1,k_2;II}^{(n-j-k_2+1)}+k_2\right)\right),\\
& n\ge k_2,
\end{aligned}
\end{equation}
with the boundary condition
\begin{equation}\label{k1k2recurbc}
M_{k_1,k_2;II}^{(n)}=\begin{cases} M_{k_1}^{(n)},\ n=k_1,\cdots, k_2-1;\\ 0,\ l=0,\cdots, k_1-1,\end{cases}
\end{equation}
where we recall that $M_{k_1}^{(n)}$ denotes the expected number of bonded molecules for a row of $n$ molecules under $k_1$-bonding.
Also, from the definition of Model II $k_1,k_2$-bonding, it follows that
 $M_{k_1,k_2;II_1}^{(n)}$, the expected value of the number of bonded molecules from a row of $n$ molecules at the end of stage one when there are no longer any unbonded $k_2$-tuples, satisfies the very same equation \eqref{k1k2recur}, but this time  with boundary condition
\begin{equation*}
M_{k_1,k_2;II_1}^{(n)}= 0,\ l=0,\cdots, k_2-1.
\end{equation*}
Thus, in light of the linearity of equation \eqref{k1k2recur}, it suffices to prove only \eqref{k1k2II} and \eqref{B(s)}, where in \eqref{k1k2II} we understand the term $m_{k_1,k_2;II_1}$ to be the
right hand side of   \eqref{k1k2II1}.

Multiplying both sides of \eqref{k1k2recur} by $(2n-k_1-k_2+2)t^n$, and summing over $n$ from $k_2$ to $\infty$, we obtain
\begin{equation}\label{multtnagain}
\begin{aligned}
&\sum_{n=k_2}^\infty(2n-k_1-k_2+2)M_{k_1,k_2;II}^{(n)}t^n=\sum_{n=k_2}^\infty\left(k_1(n-k_1+1)+k_2(n-k_2+1)\right)t^n+\\
&2\sum_{n=k_2}^\infty \left(\sum_{j=0}^{n-k_1}M_{k_1,k_2;II}^{(j)}\right)t^n+2\sum_{n=k_2}^\infty \left(\sum_{j=0}^{n-k_2}M_{k_1,k_2;II}^{(j)}\right)t^n.
\end{aligned}
\end{equation}

Define the generating function
\begin{equation}\label{mgenfunc}
m_{k_1,k_2}(t)=\sum_{n=k_2}^\infty M_{k_1,k_2;II}^{(n)}t^n.
\end{equation}
Consider the  first term on the right hand side of \eqref{multtnagain}.
A direct calculation gives
\begin{equation}\label{2sums}
\begin{aligned}
&\sum_{n=k_2}^\infty(n-k_1+1)t^n=\frac{\big[(k_2-k_1+1)-(k_2-k_1)t\big]t^{k_2}}{(1-t)^2};\\
&\sum_{n=k_2}^\infty(n-k_2+1)t^n=\frac{t^{k_2}}{(1-t)^2}.
\end{aligned}
\end{equation}
Now consider the second term on the right hand side of \eqref{multtnagain}. Using the boundary condition \eqref{k1k2recurbc}, similar to \eqref{doublesum} we have
\begin{equation}\label{doublesumMk1}
\begin{aligned}
&\sum_{n=k_2}^\infty \left(\sum_{j=0}^{n-k_1}M_{k_1,k_2;II}^{(j)}\right)t^n=
\sum_{n=\max(k_2,2k_1)}^{k_2+k_1-1}\left(\sum_{j=k_1}^{n-k_1}M_{k_1}^{(j)}\right)t^n+\sum_{n=k_2+k_1}^\infty\left(\sum_{j=k_1}^{k_2-1}M_{k_1}^{(j)}\right)t^n+\\
&\sum_{n=k_2+k_1}^\infty\left(\sum_{j=k_2}^{n-k_1}M_{k_1,k_2;II}^{(j)}\right)t^n.
\end{aligned}
\end{equation}
Note that the second term on the right hand side of \eqref{doublesumMk1} can be written as
\begin{equation}\label{2ndtermk1}
\sum_{n=k_2+k_1}^\infty\left(\sum_{j=k_1}^{k_2-1}M_{k_1}^{(j)}\right)t^n=
\left(\sum_{j=k_1}^{k_2-1}M_{k_1}^{(j)}\right)\frac{t^{k_2+k_1}}{1-t}.
\end{equation}
The final term on the right hand side of \eqref{doublesumMk1} can be dealt with  similar to \eqref{tnfree}. We have
\begin{equation}\label{tnfreeMk1}
t^{k_1}\thinspace\frac{m_{k_1,k_2}(t)}{1-t}=t^{k_1}\left(\sum_{r=0}^\infty t^r\right)\left(\sum_{j=k_2}^\infty M_{k_1,k_2;II}^{(j)}t^j\right)=\sum_{n=k_2+k_1}^\infty\left(\sum_{j=k_2}^{n-k_1}M_{k_1,k_2;II}^{(j)}\right)t^n.
\end{equation}
Now consider the third and final term on the right hand side of \eqref{multtnagain}. Similar to the above, we have
\begin{equation}\label{doublesumMk2}
\begin{aligned}
&\sum_{n=k_2}^\infty \left(\sum_{j=0}^{n-k_2}M_{k_1,k_2;II}^{(j)}\right)t^n=
\sum_{n=k_2+k_1}^{2k_2-1}\left(\sum_{j=k_1}^{n-k_2}M_{k_1}^{(j)}\right)t^n+\sum_{n=2k_2}^\infty\left(\sum_{j=k_1}^{k_2-1}M_{k_1}^{(j)}\right)t^n+\\
&\sum_{n=2k_2}^\infty\left(\sum_{j=k_2}^{n-k_2}M_{k_1,k_2;II}^{(j)}\right)t^n,
\end{aligned}
\end{equation}
and note that the second term on the right hand side of \eqref{doublesumMk2} can be written as
\begin{equation}\label{2ndtermk2}
\sum_{n=2k_2}^\infty\left(\sum_{j=k_1}^{k_2-1}M_{k_1}^{(j)}\right)t^n=
\left(\sum_{j=k_1}^{k_2-1}M_{k_1}^{(j)}\right)\frac{t^{2k_2}}{1-t},
\end{equation}
and the final term on the right hand side of  \eqref{doublesumMk2} can be written as
\begin{equation}\label{tnfreeMk2}
t^{k_2}\thinspace\frac{m_{k_1,k_2}(t)}{1-t}=t^{k_2}\left(\sum_{r=0}^\infty t^r\right)\left(\sum_{j=k_2}^\infty M_{k_1,k_2;II}^{(j)}t^j\right)=\sum_{n=2k_2}^\infty\left(\sum_{j=k_2}^{n-k_2}M_{k_1,k_2;II}^{(j)}\right)t^n.
\end{equation}
From \eqref{multtnagain}-\eqref{tnfreeMk2}, we conclude that
\begin{equation*}
\begin{aligned}
&2tm_{k_1,k_2}'(t)-(k_1+k_2-2)m_{k_1,k_2}(t)=\frac{\big[(k_2-k_1+1)-(k_2-k_1)t\big]k_1t^{k_2}}{(1-t)^2}+\frac{k_2t^{k_2}}{(1-t)^2}+\\
&2\sum_{n=\max(k_2,2k_1)}^{k_2+k_1-1}\left(\sum_{j=k_1}^{n-k_1}M_{k_1}^{(j)}\right)t^n+2\left(\sum_{j=k_1}^{k_2-1}M_{k_1}^{(j)}\right)\frac{t^{k_2+k_1}}{1-t}+
2t^{k_1}\thinspace\frac{m_{k_1,k_2}(t)}{1-t}+\\
&2\sum_{n=k_2+k_1}^{2k_2-1}\left(\sum_{j=k_1}^{n-k_2}M_{k_1}^{(j)}\right)t^n+2\left(\sum_{j=k_1}^{k_2-1}M_{k_1}^{(j)}\right)\frac{t^{2k_2}}{1-t}+
2t^{k_2}\thinspace\frac{m_{k_1,k_2}(t)}{1-t},
\end{aligned}
\end{equation*}
which we write as
\begin{equation}\label{odeagain}
\begin{aligned}
&m_{k_1,k_2}'(t)=A(t)m_{k_1,k_2}(t)+B_1(t)+B_2(t),\ \text{where}\\
&A(t)=\frac{k_1+k_2-2}{2t}+\frac{t^{k_1-1}+t^{k_2-1}}{1-t};\\
&B_1(t)=\big[k_2+k_1(k_2-k_1+1)-k_1(k_2-k_1)t\big]\frac{t^{k_2-1}}{2(1-t)^2};\\
&B_2(t)=\sum_{n=\max(k_2,2k_1)}^{k_2+k_1-1}\big(\sum_{j=k_1}^{n-k_1}M_{k_1}^{(j)}\big)t^{n-1}+\sum_{n=k_2+k_1}^{2k_2-1}\big(\sum_{j=k_1}^{n-k_2}M_{k_1}^{(j)}\big)t^{n-1}+\\
&\left(\sum_{j=k_1}^{k_2-1}M_{k_1}^{(j)}\right)\frac{t^{k_2+k_1-1}+t^{2k_2-1}}{1-t},
\end{aligned}
\end{equation}
the dependence of $A$, $B_1$ and $B_2$ on $k_1$ and $k_2$ being suppressed.
Solving the linear ODE gives for any $\epsilon\in(0,1)$,
\begin{equation}\label{odesoluagain}
m_{k_1,k_2}(t)=m_{k_1,k_2}(\epsilon)e^{\int_\epsilon^tA(s)ds}+\int_\epsilon^te^{\int_s^tA(r)dr}\left(B_1(s)+B_2(s)\right)ds,\ \epsilon<t<1.
\end{equation}

Writing $\frac{t^l}{1-t}=\frac1{1-t}-\sum_{j=0}^{l-1}t^j$, we have
$
\int\frac{t^l}{1-t}dt=-\log(1-t)-\sum_{j=1}^l\frac{t^j}j.
$
Using this with the formula for $A$ in \eqref{odeagain}, we obtain
$$
\int A(r)dr=\frac{k_1+k_2-2}2\log t-2\log(1-t)-\sum_{j=1}^{k_1-1}\frac{t^j}j-\sum_{j=1}^{k_2-1}\frac{t^j}j,
$$
and thus
\begin{equation}\label{eAformagain}
e^{\int_s^tA(r)dr}=\left(\frac ts\right)^{\frac12(k_1+k_2-2)}\frac{(1-s)^2}{(1-t)^2}\exp\left(\sum_{j=1}^{k_1-1}\frac{s^j-t^j}j+\sum_{j=1}^{k_2-1}\frac{s^j-t^j}j\right).
\end{equation}
By \eqref{mgenfunc}, $m_{k_1,k_2}(\epsilon)=O(\epsilon^{k_2})$, thus it follows from     \eqref{eAformagain} that
$\lim_{\epsilon\to0}m_{k_1,k_2}(\epsilon)e^{\int_\epsilon^tA(s)ds}=0$. Thus, from \eqref{odeagain}-\eqref{eAformagain} we conclude that
\begin{equation}\label{m(t)}
\begin{aligned}
&m_{k_1,k_2}(t)=
 \frac{t^{\frac12(k_1+k_2-2)}}{(1-t)^2}\times\\
 &\int_0^t\frac{(1-s)^2}{s^{\frac12(k_1+k_2-2)}}
\exp\left(\sum_{j=1}^{k_1-1}\frac{s^j-t^j}j+\sum_{j=1}^{k_2-1}\frac{s^j-t^j}j \right)\left(B_1(s)+B_2(s)\right)ds.
\end{aligned}
\end{equation}
Similar to the proof of  Theorem \ref{gapsthm}, we conclude from \eqref{m(t)} that
\begin{equation}\label{limMnn}
\begin{aligned}
&\lim_{n\to\infty}\frac{M_{k_1,k_2;II}^{(n)}}n=\lim_{t\to1^-}(1-t)^2m_{k_1,k_2}(t)=\\
&\int_0^1\frac{(1-s)^2}{s^{\frac12(k_1+k_2-2)}}
\exp\left(\sum_{j=1}^{k_1-1}\frac{s^j-1}j+\sum_{j=1}^{k_2-1}\frac{s^j-1}j \right)\left(B_1(s)+B_2(s)\right)ds,
\end{aligned}
\end{equation}
where $B_1,B_s$ are as in \eqref{odeagain}. Substituting for $B_1$ from \eqref{odeagain}, we have after some algebra,
\begin{equation}\label{B1term}
\begin{aligned}
&\int_0^1\frac{(1-s)^2}{s^{\frac12(k_1+k_2-2)}}
\exp\left(\sum_{j=1}^{k_1-1}\frac{s^j-1}j+\sum_{j=1}^{k_2-1}\frac{s^j-1}j \right)B_1(s)ds=\\
&\frac12\int_0^1s^{\frac{k_2-k_1}2}\big(k_1+k_2+k_1(k_2-k_1)(1-s)\big)\exp\left(\sum_{j=1}^{k_1-1}\frac{s^j-1}j+\sum_{j=1}^{k_2-1}\frac{s^j-1}j\right) ds.
\end{aligned}
\end{equation}
One can easily verify  that $\frac{(1-s)^2}{s^{\frac12(k_1+k_2-2)}}B_2(s)$, where $B_2$ is as in \eqref{odeagain}, is equal to $B(s)$, where $B$ is as in \eqref{B(s)}. Using this with
\eqref{limMnn} and \eqref{B1term} proves Theorem \ref{k1k2IIthm}.
\hfill $\square$

\section{Proof of Theorem \ref{mk1k2inftyIthm}}\label{mk1k2inftyIthmpf}
We write the second term on the right hand side of \eqref{k1k2I}
as
\begin{equation}\label{prepscale}
\sum_{l=k_1}^{k_2-1}g_{k_2;l}M_{k_1}^{(l)}=m_{k_2}\frac{k_2-1}{m_{k_2}}\sum_{l:\frac l{k_2-1}\in[\frac{k_1}{k_2-1},1]}\frac l{k_2-1}g_{k_2;l}\frac{M_{k_1}^{(l)}}l.
\end{equation}
By \eqref{muk}, Theorem \ref{weakconvgapsthm} and Corollary \ref{expvalf},
\begin{equation}\label{wkconvinproof}
\lim_{k_2\to\infty} \frac{k_2-1}{m_{k_2}}\sum_{l:\frac l{k_2-1}\in[\frac{k_1}{k_2-1},1]}g_{k_2;l}\frac l{k_2-1}=\int_0^1\gamma f_{\mu^{\text{gaps}}_\infty}(\gamma)d\gamma=\frac{1-m_\infty}{m_\infty}.
\end{equation}
By \eqref{mk}, $\lim_{l\to\infty}\frac{M_{k_1}^{(l)}}l=m_{k_1}$, and of course $\lim_{k_2\to\infty}m_{k_2}=m_\infty$. Using this with \eqref{prepscale}, \eqref{wkconvinproof} and \eqref{k1k2I} proves
Theorem \ref{mk1k2inftyIthm}.
\hfill $\square$

\section{Proof of Theorem \ref{mk1k2inftyIIthm}}\label{mk1k2inftyIIthmpf}
We break up the formula for $m_{k_1,k_2;II}$ in \eqref{k1k2II} into four parts.
The first part is $m_{k_1,k_2;II_1}$ from \eqref{k1k2II1}:
$$
\begin{aligned}
&I_{k_2}:=m_{k_1,k_2;II_1}=\\
&\frac12\int_0^1\exp\left(\sum_{j=1}^{k_1-1}\frac{s^j-1}j+\sum_{j=1}^{k_2-1}\frac{s^j-1}j\right)s^{\frac{k_2-k_1}2}\big(k_1+k_2+k_1(k_2-k_1)(1-s)\big) ds.
\end{aligned}
$$
Recalling \eqref{B(s)}, we define the   three other parts by
$$
II_{k_2}=\big(\sum_{j=k_1}^{k_2-1}M^{(j)}_{k_1}\big)\int_0^1\exp\left(\sum_{j=1}^{k_1-1}\frac{s^j-1}j+\sum_{j=1}^{k_2-1}\frac{s^j-1}j\right)
(1-s)\left(s^{\frac12k_2+\frac12k_1}+s^{\frac32k_2-\frac12k_1}\right)ds,
$$
$$
\begin{aligned}
&III_{k_2}=\int_0^1\exp\left(\sum_{j=1}^{k_1-1}\frac{s^j-1}j+\sum_{j=1}^{k_2-1}\frac{s^j-1}j\right)(1-s)^2
\sum_{n=k_2}^{k_2+k_1-1}\big(\sum_{j=k_1}^{n-k_1}M^{(j)}_{k_1}\big)s^{n-\frac{k_1}2-\frac{k_2}2}ds,
\end{aligned}
$$
and
$$
\begin{aligned}
&IV_{k_2}=\int_0^1\exp\left(\sum_{j=1}^{k_1-1}\frac{s^j-1}j+\sum_{j=1}^{k_2-1}\frac{s^j-1}j\right)(1-s)^2
\sum_{n=k_2+k_1}^{2k_2-1}\big(\sum_{j=k_1}^{n-k_2}M^{(j)}_{k_1}\big)s^{n-\frac{k_1}2-\frac{k_2}2}ds.
\end{aligned}
$$
(The lower bound in the outer summation in $III_{k_2}$ is actually $\max(k_2,2k_1)$, but since we are analyzing the case that $k_1$ is fixed and $k_2\to\infty$, it suffices to write $k_2$.)

We first consider $I_{k_2}$. Making the same substitutions as were made in
\eqref{sj-1} and \eqref{minftyderiv}, with $k_2$ in place of $k$ so that $x=k_2(1-s)$, we have
\begin{equation}\label{substk2}
\begin{aligned}
&k_2\int_a^1\exp\big(\sum_{j=1}^{k_2-1}\frac{s^j-1}j\big)s^{\frac{k_2}2}ds=\\
&\int_0^{(1-a)k_2}\exp\left(-\int_0^x\frac{1-(1-\frac y{k_2})^{k_2-1}}ydy\right)(1-\frac x{k_2})^\frac{k_2}2dx\stackrel{k_2\to\infty}{\to}\\
&\int_0^\infty\exp\left(-\int_0^x\frac{1-e^{-y}}ydy\right)e^{-\frac x2}dx,\ \text{for any}\ a\in[0,1),
\end{aligned}
\end{equation}
where the limit follows from the dominated convergence theorem.
From \eqref{substk2} it follows that
\begin{equation}\label{h(1)}
\begin{aligned}
&\lim_{k_2\to\infty}k_2\int_0^1h(s)\exp\big(\sum_{j=1}^{k_2-1}\frac{s^j-1}j\big)s^{\frac{k_2}2}ds=h(1)\int_0^\infty\exp\left(-\int_0^x\frac{1-e^{-y}}ydy\right)e^{-\frac x2}\thinspace dx,\\
&\text{for any continuous function}\ h\ \text{on}\ [0,1].
\end{aligned}
\end{equation}
From \eqref{h(1)} and the definition of $I_{k_2}$ it follows that
\begin{equation}\label{Ilimit}
\begin{aligned}
&\lim_{k_2\to\infty}I_{k_2}=\frac12\lim_{k_2\to\infty}k_2\int_0^1\exp\big(\sum_{j=1}^{k_2-1}\frac{s^j-1}j\big)s^{\frac{k_2}2}ds=\\
&\frac12\int_0^\infty\exp\left(-\int_0^x\frac{1-e^{-y}}ydy\right)e^{-\frac x2}\thinspace dx=D,
\end{aligned}
\end{equation}
where $D$ is as in \eqref{D}.
This proves \eqref{mk1k2inftyII_1} and \eqref{D}.

To prove \eqref{mk1k2inftyII} and thereby complete the proof of the theorem, we need to show that
\begin{equation}\label{234}
\lim_{k_2\to\infty}\left(II_{k_2}+III_{k_2}+IV_{k_2}\right)=(1-D)m_{k_1}.
\end{equation}
We first consider $II_{k_2}$.
Similar to \eqref{substk2}, we have
\begin{equation}\label{substk2again}
\begin{aligned}
&k_2^2\int_a^1\exp\big(\sum_{j=1}^{k_2-1}\frac{s^j-1}j\big)(1-s)s^{bk_2}ds=\\
&\int_0^{(1-a)k_2}\exp\left(-\int_0^x\frac{1-(1-\frac y{k_2})^{k_2-1}}ydy\right)x(1-\frac x{k_2})^{bk_2}dx\stackrel{k_2\to\infty}{\to}\\
&\int_0^\infty\exp\left(-\int_0^x\frac{1-e^{-y}}ydy\right)xe^{-bx}dx,\ \text{for}\ b>0\ \text{and any}\ a\in[0,1),
\end{aligned}
\end{equation}
from which it follows that
\begin{equation}\label{h(1)again}
\begin{aligned}
&\lim_{k_2\to\infty}k_2^2\int_0^1h(s)\exp\big(\sum_{j=1}^{k_2-1}\frac{s^j-1}j\big)(1-s)s^{bk_2}ds=\\
&h(1)\int_0^\infty\exp\left(-\int_0^x\frac{1-e^{-y}}ydy\right)xe^{-bx} dx,\\
& \text{for}\ b>0\ \text{and any continuous function}\ h\ \text{on}\ [0,1].
\end{aligned}
\end{equation}
By \eqref{mk}, $\sum_{j=k_1}^{k_2-1}M^{(j)}_{k_1}\sim \sum_{j=k_1}^{k_2-1}jm_{k_1}\sim m_{k_1}\frac{k_2^2}2$, as $k_2\to\infty$.
Using this with  \eqref{h(1)again} and the definition of $II_{k_2}$, it follows that
\begin{equation}\label{IIlimit}
\begin{aligned}
&\lim_{k_2\to\infty}II_{k_2}=\frac{m_{k_1}}2\int_0^\infty\exp\left(-\int_0^x\frac{1-e^{-y}}ydy\right)x\left(e^{-\frac12x}+e^{-\frac32x}\right) dx.
\end{aligned}
\end{equation}

We now consider $III_{k_2}$.
By \eqref{mk}, $\sum_{j=k_1}^{n-k_1}M^{(j)}_{k_1}\sim\sum_{j=k_1}^{n-k_1}jm_{k_1}\sim m_{k_1}\frac{k_2^2}2$ as $k_2\to\infty$, uniformly over
$n\in\{k_2,\cdots, k_2+k_1-1\}$. Thus, we have
$$
\begin{aligned}
&\sum_{n=k_2}^{k_2+k_1-1}\big(\sum_{j=k_1}^{n-k_1}M^{(j)}_{k_1}\big)s^{n-\frac{k_1}2-\frac{k_2}2}\sim k_1m_{k_1}\frac{k_2^2}2s^{\frac{k_2}2}h_{k_1}(s),\ \text{as}\ k_2\to\infty,\\
&\text{where}\ h_{k_1}(s)\ \text{is continuous on}\  [0,1].
\end{aligned}
$$
Thus,
\begin{equation}\label{estIII}
\begin{aligned}
&III_{k_2}\sim k_1m_{k_1}\frac{k_2^2}2\int_0^1\exp\left(\sum_{j=1}^{k_1-1}\frac{s^j-1}j+\sum_{j=1}^{k_2-1}\frac{s^j-1}j\right)(1-s)^2s^{\frac{k_2}2}h_{k_1}(s)ds.
\end{aligned}
\end{equation}
Making the same substitutions as before, we have
\begin{equation}\label{estIIIsubst}
\begin{aligned}
&k_2^2\int_0^1\exp\left(\sum_{j=1}^{k_2-1}\frac{s^j-1}j\right)(1-s)^2s^{\frac{k_2}2}ds=\\
&\frac1{k_2}\int_0^{k_2}\exp\left(-\int_0^x\frac{1-(1-\frac y{k_2})^{k_2-1}}ydy\right)x^2(1-\frac x{k_2})^{\frac{k_2}2}dx\stackrel{k_2\to\infty}{\to}0.
\end{aligned}
\end{equation}
From \eqref{estIII} and \eqref{estIIIsubst}, we conclude that
\begin{equation}\label{IIIlimit}
\lim_{k_2\to\infty}III_{k_2}=0.
\end{equation}

We now consider $IV_{k_2}$, which is somewhat more involved.  From the substitutions in
\eqref{sj-1} and \eqref{minftyderiv} that were applied to obtain
 \eqref{substk2} and \eqref{substk2again}, and from  the conclusions in \eqref{h(1)} and \eqref{h(1)again}, it is easy to see that the following holds.
\begin{equation}\label{expression}
\begin{aligned}
&\text{For}\ h(x)\ \text{continuous on}\ [0,1]\  \text{with}\ h(1)>0\ \text{and}\ b>0,\\
&\lim_{k_2\to\infty}k^C_2\int_0^1\exp\big(\sum_{j=1}^{k_2-1}\frac{s^j-1}j\big)(1-s)^Ds^{bk_2}h(s)ds=\\
&\begin{cases} h(1)\int_0^\infty\exp\left(-\int_0^x\frac{1-e^{-y}}ydy\right)x^De^{-bx}dx,\ \text{if}\ C=D+1;\\
\infty,\ \text{if}\ C>D+1;\\
0,\ \text{if}\ C<D+1.\end{cases}
\end{aligned}
\end{equation}
Because of this, when calculating $\lim_{k_2\to\infty}IV_{k_2}$, we may replace the term $M_{k_1}^{(j)}$  by the expression for its leading asymptotic behavior as
$j\to\infty$, namely $jm_{k_1}$, as follows from \eqref{mk}.
We can also let $h(s)=s^{-\frac{k_1}2}\exp(\sum_{j=1}^{k_1-1}\frac{s^j-1}j)$, which satisfies $h(1)=1$. Thus,
we have
\begin{equation}\label{replacementasymp}
\begin{aligned}
&\lim_{k_2\to\infty}IV_{k_2}=\\
&\lim_{k_2\to\infty}\int_0^1\exp(\sum_{j=1}^{k_2-1}\frac{s^j-1}j)(1-s)^2
\sum_{n=k_2+k_1}^{2k_2-1}\big(\sum_{j=k_1}^{n-k_2}jm_{k_1}\big)s^{n-\frac{k_2}2}ds.
\end{aligned}
\end{equation}
The same considerations  in fact allow us to conclude that
\begin{equation}\label{replacementasympagain}
\begin{aligned}
&\lim_{k_2\to\infty}IV_{k_2}=\\
&\lim_{k_2\to\infty}\int_0^1\exp(\sum_{j=1}^{k_2-1}\frac{s^j-1}j)(1-s)^2
\sum_{n=k_2}^{2k_2-1}\big(\sum_{j=0}^{n-k_2}jm_{k_1}\big)s^{n-\frac{k_2}2}ds,
\end{aligned}
\end{equation}
because the difference between $\sum_{n=k_2}^{2k_2-1}\big(\sum_{j=0}^{n-k_2}jm_{k_1}\big)s^{n-\frac{k_2}2}$ and\newline $\sum_{n=k_2+k_1}^{2k_2-1}\big(\sum_{j=k_1}^{n-k_2}jm_{k_1}\big)s^{n-\frac{k_2}2}$ is of lower order than
$\sum_{n=k_2+k_1}^{2k_2-1}\big(\sum_{j=k_1}^{n-k_2}jm_{k_1}\big)s^{n-\frac{k_2}2}$ as $k_2\to\infty$.
We prefer to work with \eqref{replacementasympagain} rather than with \eqref{replacementasymp} because this makes the calculations below somewhat simpler.

We write
\begin{equation}\label{suminintegral}
\sum_{n=k_2}^{2k_2-1}\big(\sum_{j=0}^{n-k_2}jm_{k_1}\big)s^{n-\frac{k_2}2}=\frac{m_{k_1}}2\sum_{n=k_2}^{2k_2-1}(n-k_2)(n-k_2+1)s^{n-\frac{k_2}2},
\end{equation}
and
\begin{equation}\label{contsuminintegral}
\begin{aligned}
&\sum_{n=k_2}^{2k_2-1}(n-k_2)(n-k_2+1)s^{n-\frac{k_2}2}=s^{\frac{k_2}2}\sum_{m=0}^{k_2-1}m(m+1)s^m=\\
&s^{\frac{k_2}2}\sum_{m=-1}^{k_2-1}m(m+1)s^m=
s^{\frac{k_2}2+1}\big(\sum_{m=-1}^{k_2-1}s^{m+1}\big)''=s^{\frac{k_2}2+1}\big(\frac{1-s^{k_2+1}}{1-s}\big)''=\\
&s^{\frac{k_2}2+1}\left(\frac{2(1-s^{k_2+1})}{(1-s)^3}-\frac{2(k_2+1)s^{k_2}}{(1-s)^2}-\frac{k_2(k_2+1)s^{k_2-1}}{1-s}  \right).
\end{aligned}
\end{equation}
From \eqref{replacementasympagain}-\eqref{contsuminintegral}
we have
\begin{equation}\label{replacementasympagainagain}
\begin{aligned}
&\lim_{k_2\to\infty}IV_{k_2}=\frac{m_{k_1}}2\lim_{k_2\to\infty}\int_0^1\exp(\sum_{j=1}^{k_2-1}\frac{s^j-1}j)\times\\
&\left(\frac{2(s^{\frac{k_2}2+1}-s^{\frac{3k_2}2+2})}{(1-s)}-2(k_2+1)s^{\frac{3k_2}2+1}-k_2(k_2+1)(1-s)s^{\frac{3k_2}2}\right)ds.
\end{aligned}
\end{equation}
Making the same substitutions as before, we obtain
\begin{equation}\label{substIV}
\begin{aligned}
&\lim_{k_2\to\infty}\int_0^1\exp(\sum_{j=1}^{k_2-1}\frac{s^j-1}j)\times\\
&\left(\frac{2(s^{\frac{k_2}2+1}-s^{\frac{3k_2}2+2})}{(1-s)}-2(k_2+1)s^{\frac{3k_2}2+1}-k_2(k_2+1)(1-s)s^{\frac{3k_2}2}\right)ds=\\
&\int_0^\infty\exp\left(-\int_0^x\frac{1-e^{-y}}ydy\right)\left( \frac{2(e^{-\frac x2}-e^{-\frac{3x}2})}x-2e^{-\frac{3x}2}-xe^{-\frac{3x}2}\right)dx.
\end{aligned}
\end{equation}
From \eqref{replacementasympagainagain} and \eqref{substIV} we conclude that
\begin{equation}\label{IVlimit}
\lim_{k_2\to\infty}IV_{k_2}=\frac{m_{k_1}}2\int_0^\infty\exp\left(-\int_0^x\frac{1-e^{-y}}ydy\right)\left( \frac{2(e^{-\frac x2}-e^{-\frac{3x}2})}x-2e^{-\frac{3x}2}-xe^{-\frac{3x}2}\right)dx.
\end{equation}

Note that
$$
\left(\exp\left(-\int_0^x\frac{1-e^{-y}}ydy\right)xe^{-\frac12x}\right)'=-\exp\left(-\int_0^x\frac{1-e^{-y}}ydy\right)\left(\frac12x e^{-\frac12x}-e^{-\frac32x}\right).
$$
Thus,
\begin{equation}\label{integral=0}
\begin{aligned}
&\int_0^\infty\exp\left(-\int_0^x\frac{1-e^{-y}}ydy\right)\left(\frac12x e^{-\frac12x}-e^{-\frac32x}\right)dx=\\
&-\left(\exp\left(-\int_0^x\frac{1-e^{-y}}ydy\right)xe^{-\frac12x}\right)\Big|_0^\infty=0.
\end{aligned}
\end{equation}
Also note
that the term
$\frac{m_{k_1}}2\int_0^\infty\exp\left(-\int_0^x\frac{1-e^{-y}}ydy\right)xe^{-\frac32x}dx$  in \eqref{IIlimit} cancels with a term from \eqref{IVlimit}.
Using this last fact along with \eqref{IIlimit}, \eqref{IIIlimit}, \eqref{IVlimit} and \eqref{integral=0}, we obtain
\begin{equation}\label{II+III+IV}
\lim_{k_2\to\infty}\left(II_{k_2}+III_{k_2}+IV_{k_2}\right)=m_{k_1}\int_0^\infty\exp\left(-\int_0^x\frac{1-e^{-y}}ydy\right)\frac{(e^{-\frac x2}-e^{-\frac{3x}2})}xdx.
\end{equation}
Integrating by parts with $u=e^{-\frac12x}$ and $v'=\exp\left(-\int_0^x\frac{1-e^{-y}}ydy\right)\frac{1-e^{-x}}x$,
we have
\begin{equation}\label{finalbyparts}
\begin{aligned}
&\int_0^\infty\exp\left(-\int_0^x\frac{1-e^{-y}}ydy\right)\frac{(e^{-\frac x2}-e^{-\frac{3x}2})}xdx=\\
&-e^{-\frac12x}\exp\left(-\int_0^x\frac{1-e^{-y}}ydy\right)\Big|_0^\infty-\frac12\int_0^\infty\exp\left(-\int_0^x\frac{1-e^{-y}}ydy\right)e^{-\frac12x}dx=\\
&1-\frac12\int_0^\infty\exp\left(-\int_0^x\frac{1-e^{-y}}ydy\right)e^{-\frac12x}dx.
\end{aligned}
\end{equation}
Now \eqref{234} follows from \eqref{II+III+IV} and \eqref{finalbyparts}.

\hfill $\square$

\section{Proof of Theorem \ref{k1k1LIinftythm}}\label{k1k1LIinftythmpf}
Recalling that $M_{k_1}^{(l)}=k_1$, for $l\in[k_1,2k_1-1]$, we have from
 \eqref{k1k2I},
\begin{equation}\label{mk1Lk1form}
m_{k_1,[Lk_1];I}=m_{[Lk_1]}+\sum_{l=k_1}^{[Lk_1]-1}g_{[Lk_1];l}M^{(l)}_{k_1}=m_{[Lk_1]}+k_1\sum_{l=k_1}^{[Lk_1]-1}g_{[Lk_1];l},\ L\in(1,2].
\end{equation}
Substituting from \eqref{lgaps} in the sum on the right hand side of \eqref{mk1Lk1form}, we have
\begin{equation}\label{mk1Lk1formsum}
\begin{aligned}
&k_1\sum_{l=k_1}^{[Lk_1]-1}g_{[Lk_1];l}=2k_1\int_0^1(1-s)\left(\sum_{l=k_1}^{[Lk_1]-1}s^l\right)e^{2\sum_{j=1}^{[Lk_1]-1}\frac{s^j-1}j}ds=\\
&2k_1\int_0^1\left(s^{k_1}-s^{[Lk_1]}\right)e^{2\sum_{j=1}^{[Lk_1]-1}\frac{s^j-1}j}ds.
\end{aligned}
\end{equation}
Making the same substitutions as were made in
\eqref{sj-1} and \eqref{minftyderiv}, with $[Lk_1]$ in place of $k$ so that $x=[Lk_1](1-s)$,
we  obtain from  \eqref{mk1Lk1formsum}
\begin{equation}
\begin{aligned}
&k_1\sum_{l=k_1}^{[Lk_1]-1}g_{[Lk_1];l}=\\
&\frac2L\int_0^{[Lk_1]}\left((1-\frac x{[Lk_1]})^{k_1}-(1-\frac x{[Lk_1]})^{[Lk_1]}\right)\exp\left(-2\int_0^x\frac{1-(1-\frac y{[Lk_1]})^{[Lk_1]-1}}ydy\right)dx\stackrel{k_1\to\infty}{\to}\\
&\frac2L\int_0^\infty \left(e^{-\frac1L x}-e^{-x}\right) e^{-2\int_0^x\frac{1-e^{-y}}ydy}dx,
\end{aligned}
\end{equation}
which proves \eqref{k1k1LIinfty}.
\hfill $\square$

\section{Proof of Theorem \ref{k1k1LIIinftythm}}\label{k1k1LIIinftythmpf}
We begin by writing down explicitly $B(s)$ from  \eqref{B(s)} in the case that $k_2=[Lk_1]$ with $L\in(1,2]$.  Since $L\in(1,2]$, we have $M_{k_1}^{(j)}=k_1$ in every place it appears on the right hand side of \eqref{B(s)}. We continue to write $k_2$ instead of $[Lk_1]$, in order to have  the long calculations below take up a little less space. We have
\begin{equation}\label{BLle2}
\begin{aligned}
&B(s)=(1-s)^2s^{-\frac{k_1}2-\frac{k_2}2}\times\\
&\left(k_1\sum_{n=2k_1}^{k_2+k_1-1}(n-2k_1+1)s^n+k_1\sum_{n=k_2+k_1}^{2k_2-1}\big(n-k_2-k_1+1\big)s^n\right)+\\
&k_1\left(k_2-k_1\right)(1-s)\left(s^{\frac{3k_2}2-\frac{k_1}2}+s^{\frac{k_2}2+\frac{k_1}2} \right).
\end{aligned}
\end{equation}
Differentiating and performing some algebra, we have
\begin{equation}\label{ab}
\begin{aligned}
&\sum_{n=a}^bns^n=s\left(\sum_{n=a}^bs^n\right)'=\\
&\frac s{(1-s)^2}\left(as^{a-1}-(a-1)s^a-(b+1)s^b+bs^{b+1}\right),\ a<b.
\end{aligned}
\end{equation}
Using \eqref{ab}, the two sums on the right hand side of \eqref{BLle2} can be written as
\begin{equation}\label{twosums}
\begin{aligned}
&\sum_{n=2k_1}^{k_2+k_1-1}(n-2k_1+1)s^n=\frac s{(1-s)^2}\times\\
&\left(2k_1s^{2k_1-1}-(2k_1-1)s^{2k_1}-(k_2+k_1)s^{k_2+k_1-1}+(k_2+k_1-1)s^{k_2+k_1}\right)-\\
&\frac{2k_1-1}{1-s}\left(s^{2k_1}-s^{k_2+k_1}\right);\\
&\sum_{n=k_2+k_1}^{2k_2-1}(n-k_2-k_1+1)s^n=\frac s{(1-s)^2}\times\\
&\left((k_2+k_1)s^{k_2+k_1-1}-(k_2+k_1-1)s^{k_2+k_1}-2k_2s^{2k_2-1}+(2k_2-1)s^{2k_2}\right)-\\
&\frac{k_2+k_1-1}{1-s}\left(s^{k_2+k_1}-s^{2k_2}\right).
\end{aligned}
\end{equation}
Substituting \eqref{twosums} in \eqref{BLle2}, we obtain
\begin{equation*}
\begin{aligned}
&B(s)=\\
&k_1\left(2k_1s^{\frac32k_1-\frac12k_2}-(2k_1-1)s^{\frac32k_1-\frac{k_2}2+1}-(k_2+k_1)s^{\frac{k_2}2+\frac{k_1}2}+(k_2+k_1-1)s^{\frac{k_2}2+\frac{k_1}2+1}\right)-\\
&k_1(2k_1-1)(1-s)\left(s^{\frac32k_1-\frac{k_2}2}-s^{\frac{k_2}2+\frac{k_1}2}\right)+\\
&k_1\left((k_2+k_1)s^{\frac{k_2}2+\frac{k_1}2}-(k_2+k_1-1)s^{\frac{k_2}2+\frac{k_1}2+1}-2k_2s^{\frac32k_2-\frac12k_1}+(2k_2-1)s^{\frac32k_2-\frac12k_1+1}\right)-\\
&k_1(k_2+k_1-1)(1-s)\left(s^{\frac{k_2}2+\frac{k_1}2}-s^{\frac32k_2-\frac12k_1}\right)+\\
&k_1\left(k_2-k_1\right)(1-s)\left(s^{\frac{3k_2}2-\frac{k_1}2}+s^{\frac{k_2}2+\frac{k_1}2} \right).
\end{aligned}
\end{equation*}
Collecting terms, one finds that the above calculation reduces to
\begin{equation}\label{B(s)k2le2k1}
B(s)=k_1(s^{\frac32k_1-\frac12k_2}-s^{\frac32k_2-\frac12k_1}).
\end{equation}

From \eqref{k1k2II}, \eqref{k1k2II1} and \eqref{B(s)k2le2k1}, we conclude that
\begin{equation}\label{mk1k2le2k1}
\begin{aligned}
&m_{k_1,k_2;II}=\\
&\frac12\int_0^1s^{\frac{k_2-k_1}2}\big(k_1+k_2+k_1(k_2-k_1)(1-s)\big)\exp\left(\sum_{j=1}^{k_1-1}\frac{s^j-1}j+\sum_{j=1}^{k_2-1}\frac{s^j-1}j\right) ds+\\
&\int_0^1k_1(s^{\frac32k_1-\frac12k_2}-s^{\frac32k_2-\frac12k_1})\exp\left(\sum_{j=1}^{k_1-1}\frac{s^j-1}j+\sum_{j=1}^{k_2-1}\frac{s^j-1}j\right) ds,\ \text{if}\ k_2\le 2k_1.
\end{aligned}
\end{equation}
Consider the first term on the right hand side of \eqref{mk1k2le2k1} with $k_2=[Lk_1]$. By the same substitution  used in
\eqref{sj-1} and \eqref{minftyderiv},
  with $k_1$ in place of $k$ so that $x=k_1(1-s)$,
   we have
\begin{equation}\label{firsttermfinal9}
\begin{aligned}
&\lim_{k_1\to\infty}\frac12\int_0^1s^{\frac{k_2-k_1}2}\big(k_1+k_2+k_1(k_2-k_1)(1-s)\big)\exp\left(\sum_{j=1}^{k_1-1}\frac{s^j-1}j+\sum_{j=1}^{k_2-1}\frac{s^j-1}j\right) ds=\\
&\lim_{k_1\to\infty}\frac12\int_0^{k_1}(1-\frac x{k_1})^{\frac{L-1}2k_1}\left(1+L+(L-1)x\right)\times\\
&\exp\left(-\int_0^x\frac{1-(1-\frac yk)^{k-1}}ydy-\int_0^{Lx}\frac{1-(1-\frac yk)^{k-1}}ydy\right)dx=\\
&\frac12\int_0^\infty\left(L+1+(L-1)x\right)e^{-\frac{L-1}2x}e^{-\int_0^x\frac{1-e^{-y}}ydy-\int_0^{Lx}\frac{1-e^{-y}}ydy}dx.
\end{aligned}
\end{equation}
Similarly, for the second term on the right hand side of \eqref{mk1k2le2k1}, we have
\begin{equation}\label{secondtermfinal9}
\begin{aligned}
&\lim_{k_1\to\infty}\int_0^1k_1(s^{\frac32k_1-\frac12k_2}-s^{\frac32k_2-\frac12k_1})\exp\left(\sum_{j=1}^{k_1-1}\frac{s^j-1}j+\sum_{j=1}^{k_2-1}\frac{s^j-1}j\right) ds=\\
&\int_0^\infty  \left(e^{-\frac{3-L}2x}-e^{-\frac{3L-1}2x} \right)e^{-\int_0^x\frac{1-e^{-y}}ydy}e^{-\int_0^{Lx}\frac{1-e^{-y}}ydy}dx.
\end{aligned}
\end{equation}
From \eqref{mk1k2le2k1}-\eqref{secondtermfinal9}, we obtain \eqref{k1k1LIIinfty}, completing the proof of the theorem.
\hfill $\square$

\end{document}